\documentclass[12pt]{amsart}
\usepackage{amsmath}
\usepackage{amssymb} 
\usepackage{amsthm}
\usepackage{mathtools}
\usepackage{mathrsfs} 
\usepackage{dsfont} 
\usepackage{bm}
\usepackage{bbm}  
\usepackage{cite}
\usepackage{enumerate}
\usepackage{units}
\usepackage{syntonly}
\usepackage{stdclsdv}
\usepackage{alltt}
\usepackage{fontenc}
\usepackage{fancyhdr}
\usepackage{eqnarray}
\usepackage{calc}
\numberwithin{equation}{subsection}
\theoremstyle{plain}
\newtheorem{theo}{Theorem}
\newtheorem*{theo*}{Theorem}

\theoremstyle{definition}

\newtheorem{lem}{Lemma}[section]
\newtheorem{defi}[lem]{Definition}
\newtheorem{conj}[lem]{Conjecture}
\newtheorem{prob}[lem]{Problem}
\newtheorem{prop}[lem]{Proposition}

\theoremstyle{remark}
\newtheorem*{ack}{Acknowledgements}
\newtheorem*{rem*}{Remark}
\newtheorem{rem}[lem]{Reamrk}

\newtheorem{cor}[lem]{Corollary}
\begin{document}

\title{Busemann functions and barrier functions}
\author{Xiaojun Cui AND Jian Cheng}

\address{Xiaojun Cui \endgraf
Department of Mathematics, Nanjing University,
Nanjing 210093, Jiangsu Province,  People's
Republic of China.}
\email{xcui@nju.edu.cn}

\address{Jian Cheng\endgraf
Department of Mathematics, Nanjing University,
Nanjing 210093, Jiangsu Province,  People's
Republic of China.}
\email{jcheng@nju.edu.cn}


\keywords{Busemann function; Barrier function; Semi-concave; Singularity}
\subjclass[2010]{49L25, 53C22}
\thanks{The first author is supported by the National Natural Science Foundation of China
 (Grant 11271181). Both authors are supported by the Project Funded by the Priority Academic Program Development of Jiangsu Higher Education Institutions  (PAPD) and the Fundamental Research Funds for the Central Universities.}

 \abstract
We show that Busemann functions on a smooth, non-compact, complete, boundaryless, connected Riemannian manifold are viscosity solutions with respect to the Hamilton-Jacobi equation determined by the Riemannian metric and consequently they are locally semi-concave with linear modulus. We also analysis the structure of singularity sets of Busemann functions. Moreover we study  barrier functions, which are analogues to Mather's barrier functions in Mather theory, and provide some fundamental properties. Based on barrier functions, we could define some relations on the set of lines and thus classify them. We also discuss some initial relations with the ideal boundary of the Riemannian manifold.
\endabstract
\maketitle

\section*{Introduction}
Let $M$ be a smooth, non-compact, complete, boundaryless, connected  Riemannian manifold with Riemanian metric $g$. Let $d$ be the distance on $M$ and $| \cdot|_g$ the norm on the tangent bundle $TM$ and/or the contangent bundle $T^*M$  induced by this Riemannian metric $g$. Throughout this paper, all geodesic segments are always parameterized to be unit-speed.  By a ray, we mean a geodesic segment $\gamma:[0, +\infty) \rightarrow M$ such that $d(\gamma(t_1),\gamma(t_2))=|t_2-t_1|$ for any $t_1,t_2 \geq 0$. Throughout this paper, $| \cdot |$ mean Euclidean norms. To guarantee the existence of rays, $M$ must be non-compact. By definition, the Busemann function associated to a ray $\gamma$, is defined as
$$b_{\gamma}(x):=\lim_{t \rightarrow +\infty}[d(x, \gamma(t))-t].$$
Clearly, $b_{\gamma}$ is a Lipschitz function with Lipschitz constant 1, i.e.
$$|b_{\gamma}(x)-b_{\gamma}(y)| \leq d(x,y).$$

Busemann functions play an important role in the study of differential geometry. The geometrical features of the Riemannian metric will provide some properties of Busemann functions. For instance, in [\cite{Wu}, Theorem A(a)] ( also see [\cite{Sa}, Page 212, Proposition 3.1]) it is showed in our terminology that if $(M,g)$ is of non-negative sectional curvature, then $b_{\gamma}$ is a concave function on $M$ (i.e. $b_{\gamma}$ is a concave function restricted on any geodesic segment). If it is  of nonnegative Ricci curvature, then any Busemann function is superharmonic [\cite{Sa},
Page 218, Proposition 3.8], [\cite{Pe}, Page 287, Lemma 44]. There are also plenty of results on the properties of Busemann functions for the manifold of negative or non-positive curvature, or free of conjugate points or focal points (see for example \cite{Es}, \cite{HiH}).  Busemann functions are also tools  studying rigidity problem (e.g. \cite{BG}, \cite{BE1}, \cite{BE2}).

In this paper, we will study Busemann functions from the viewpoint of Mather theory. Comparing with the geometrical viewpoint, the setting in this paper are more general, at least no curvature assumptions are needed. Of course, this leads the results here are relatively rough. Our main aim is to provide a frame for the most general setting.

To state our main results, we need to recall the definition of semi-concavity with linear modulus.
\begin{defi}
Given an open subset $\Omega \subset \mathbb{R}^n$, a continuous function
$$u:\Omega \rightarrow \mathbb{R}$$
is called locally semi-concave with linear modulus if, for any open convex subset $\Omega^{\prime} \subset \Omega$ with compact support in $\Omega$ (i.e. $\Omega^{\prime} \Subset\Omega$), there exists a constant $C$ such that $u(x)-\frac{C}{2}|x|^2$ is a concave function in $\Omega^{\prime}$ (here, $| \cdot |$ is the Euclidean norm).
\end{defi}
We call the constant $C$ (depends on the choice of $\Omega^{\prime}$)  is the semi-concave constant.

For the functions defined on manifold $M$, we say
\begin{defi}
A continuous function $u:M \rightarrow \mathbb{R}$ is called locally semi-concave with linear modulus if, for any $x \in M$, there exist an open neighborhood $U$ and a smooth coordinate chart
$$\phi: U  \rightarrow \mathbb{R}^n,$$
such that the function $u \circ \phi^{-1}$ is  locally semi-concave on  $\phi(U)$.
\end{defi}

\begin{rem}
For two different charts $\phi_1,\phi_2$ both defined on $U$, $u \circ \phi^{-1}_1$ is locally semi-concave  with linear modulus if and only $u \circ \phi^{-1}_2$ is locally semi-concave  with linear modulus, although the semi-concave constants of $u \circ \phi^{-1}_1$ and of $u \circ \phi^{-1}_2 $ may be different. So the definition is well posed.
\end{rem}

\begin{rem}
For other (locally) semi-concave functions with more general modulus, a good reference is \cite{CS}.
\end{rem}
Let $\nabla$ be the gradient determined by the Riemannian metric $g$. Our first result, without any additional curvature restriction, is

\begin{theo}
Every Busemann function is  a (globally defined) viscosity solution of Hamilton-Jacobi equation $$|\nabla u|_g=1,$$
or equivalently
$$|du|_g^2=1.$$
Consequently, any Busemann function is locally semi-concave with linear modulus.
\end{theo}

\begin{rem}
The definition of viscosity solution is standard, which may be found in \cite{Li}, \cite{CS}, \cite{Fa} for example. The result of Theorem 1 is very fundamental and its proof is also quite simple, but we  could not find a literature containing such a result.
\end{rem}
\begin{rem}
Obviously, there exist such functions which are viscosity solutions but not  Busemann functions. For example on $\mathbb{R}$, $-|x|$ is a viscosity solution, but is not a Busemann function for any ray.
\end{rem}

For a ray $\gamma:[0, \infty)\rightarrow M$ and any point $x \in M$, a geodesic segment $\gamma^{\prime}:[0, +\infty) \rightarrow M$ is said to be a coray initiated from $x$ to $\gamma$ if $\gamma^{\prime}(0)=x$ and there exist a sequence $t_i \rightarrow \infty$ and a sequence of minimal geodesic segments
$$\gamma_i:[0,T_i] \rightarrow M$$
with $\gamma_i(0) \rightarrow x$ and $\gamma_i(T_i)=\gamma(t_i)$ such that $\gamma_i \rightarrow \gamma^{\prime}$ uniformly on any compact interval of $\mathbb{R}^+$. Clearly, any coray is itself a ray. The following  lemma (for a proof, see for example [\cite{BG}, Proposition 2.7]) is very useful:
 \begin{lem}
 A ray $\gamma^{\prime}$ is a coray to $\gamma$ if and only if $b_{\gamma}(\gamma^{\prime}(t_2))-b_{\gamma}(\gamma^{\prime}(t_1))=t_1-t_2$ for any $t_2,t_1 \in \mathbb{R}^+$.
  \end{lem}
  Note that, if $\gamma^{\prime}$ is a coray to $\gamma$, $\gamma$ is not a co-ray to $\gamma^{\prime}$ in general. Additionally two distinct rays $\gamma$ and $\gamma^{\prime}$ initiated from the same point $x$ may define the same Busemann functions, i.e. $b_{\gamma}=b_{\gamma^{\prime}}$. For more information on Busemann functions, we refer to \cite{Bu}.

Let singular set $sing(b_{\gamma})$ be the set of non-differentiable points of $b_{\gamma}$. Let $C_{\gamma}$ be the set of points from which at least two corays initiated. In [\cite{In}, Theorem 12] it is proved, among other results, that $C_{\gamma} \subseteq sing(b_{\gamma})$. As a  direct application of Theorem 1, we obtain
\begin{theo}
$C_{\gamma}=sing(b_{\gamma})$.
\end{theo}
Now we introduce a more restrictive notation than ray. A geodesic $\gamma: \mathbb{R} \rightarrow M$ is called to be a line if
$$d(\gamma(t_1), \gamma(t_2))=|t_2-t_1|$$
for any $t_2,t_1 \in \mathbb{R}.$ On any non-compact complete Riemannian manifold $M$, for any point $x$ on it, there always exists at least one ray initiated from $x$. Comparing with rays, lines are much more rare  and there exist examples of non-compact complete Riemannian manifolds containing no line at all. Given a line $\gamma:\mathbb{R} \rightarrow M$, we can define the barrier function $B_{\gamma}$, under the motivation of Mather theory, to be
$$B_{\gamma}(x)=b_{\gamma^{+}}(x)+b_{\gamma^{-}}(x),$$
where $\gamma^{+},\gamma^{-}$ are two rays defined by
$\gamma^{\pm}:[0, +\infty) \rightarrow M, \gamma^{\pm}(t)=\gamma(\pm t)$. It should be mentioned that such kind of barrier functions has appeared in literature (e.g. [\cite{Pe}, Page 287, Line -15--Line-10],   [\cite{Sa}, Page 218, Line -3]) in the context of non-negative Ricci curvature. But we think that the idea here, studying this function in the most general case from the viewpoint of Mather theory, is new.

For any line $\gamma:\mathbb{R} \rightarrow M$, the $\tau$-translation of $\gamma$ is  defined by  $\gamma_{\tau}(t):=\gamma(t+\tau)$ (here $\tau \in \mathbb{R}$) and $- \gamma$ is defined by $-\gamma(t):=\gamma(-t)$,  for any $t \in \mathbb{R}$.
 \begin{lem}
 1)$B_{\gamma}$ is independent of time translation, i.e. $B_{\gamma_{\tau}}=B_{\gamma}$ for any $\tau \in \mathbb{R}$.

 2)$B_{\gamma} \geq 0$, $B_{\gamma}|_{\gamma}\equiv 0$ and $B_{-\gamma}=B_{\gamma}$.
 \end{lem}
 \begin{proof}
 1) Since $b_{\gamma^+_{\tau}}=b_{\gamma^+}+\tau$ and $b_{\gamma^{-}_{\tau}}=b_{\gamma^-}-\tau$, the assertion follows.

 2) Since $\gamma$ is a line, the first statement follows  obviously from the triangle inequality. The other two assertions also hold  obviously.
 \end{proof}
 Let $G_{\gamma}=\{x:B_{\gamma}(x)=0\}$. Then, we have
 \begin{theo}
1) $B_{\gamma}$ is locally semi-concave with linear modulus.

 2) For $x \in G_{\gamma}$, there exists a unique line $\gamma^{\prime}$ through $x$ (i.e. $\gamma^{\prime}(0)=x$), such that $\gamma^{\prime}\subseteq G_{\gamma}$ and for any $\tau \in \mathbb{R}$, ${\gamma^{\prime}}^+_{\tau}$ is a coray to $\gamma^+$, ${\gamma^{\prime}}^-_{\tau}$ is a coray to $\gamma^-$. In other words, $G_{\gamma}$ is foliated by such kind of lines.

 3) For any line $\gamma$ and any compact subset $K$, there exists a constant $C$ (depends on $K$) such that $$B_{\gamma}(x) \leq Cd^2(x, \gamma)$$
 for any $x \in K$.

 \end{theo}

 \begin{rem}
 The third assertion of Theorem 3 is a variant version of Mather's result [\cite{Mat2}, Page 1375, Line 19-- Line 25] on some kind of barrier functions for positive definite Lagrangian systems.
 \end{rem}
 We denote by $\gamma^{\prime} \prec \gamma$ if $B_{\gamma}(\gamma^{\prime}(0))=0$ and $\gamma^{\prime}$ is the unique line through $\gamma^{\prime}(0)$ satisfying the second assertion of Theorem 3. In fact, $B_{\gamma}\gamma^{\prime}\equiv 0$ in this case.  We say that $\gamma^{\prime} \sim \gamma$ if $\gamma^{\prime} \prec \gamma$ and $\gamma \prec \gamma^{\prime}$.
\begin{theo}
$\prec$ is a transitive relation and thus $\sim$ is an equivalence relation. For two lines $\gamma_1, \gamma_2$ with $\gamma_1 \prec \gamma, \gamma_2 \prec \gamma$, $\gamma_1 \sim \gamma_2$ if and only if $B_{\gamma_1}=B_{\gamma_2}$.
\end{theo}

On a noncompact Riemannian manifold, one can use rays to define  a kind of  (ideal) boundary $M(\infty)$ ( see [\cite{Ba}, Section 2], \cite{Mar} for various kinds of ideal boundary). Precisely, $M(\infty)$ is the set of equivalence classes of rays, where two rays $\gamma_1$ and $\gamma_2$ are equivalent if and only if $b_{\gamma_1}-b_{\gamma_2}=const.$. Thus, for a line $\gamma_1$, we may think it as a geodesic connecting two elements (i.e. $\gamma^-_1, \gamma^+_1$) in $M(\infty)$.

\begin{theo}
For two lines $\gamma_1$ and $\gamma_2$, $\gamma_1 \sim \gamma_2$ if and only if they connect the same pair of boundary elements in $M(\infty)$.
\end{theo}

 The organization of this paper is as follows.  In section 1, we will prove that Busemann functions are viscosity solutions and show  the local semi-concavity.
 In section 2, we will
  illustrate the structure of singularity sets of Busemann functions. In section 3, Theorem 3 is proved. In section 4, we will analysis the relations $\prec$ and $\sim$ and prove Theorem 4. In section 5, we discuss the geometrical meaning of the relation $\sim$ and prove Theorem 5. In section 6, we  relate our results to a rigidity conjecture. In section 7, we propose some further discussions on the relations with ideal boundaries.

\section{Semi-concavity of Busemann functions}

 First we recall a result in \cite{MM}. For any non-empty closed subset $K$ of $M$, let $d_K(x)$ be the distance from $x$ to $K$.
 \begin{prop}[\cite{MM}, Theorem 3.1, Proposition 3.4]
 $d_K$ is locally semi-concave with linear modulus and is a viscosity solution of $|\nabla u|_g=1$ on $M\setminus K$.
 \end{prop}
 Based on this proposition, we could prove Theorem 1 as follows.
\begin{proof}[Proof of Theorem 1] We fix any ray $\gamma$ in $M$. Recall the  definition
$$b_{\gamma}(x)=\lim [d(x, \gamma(t))-t].$$
For any open subset $U \subset M$ with compact closure, there exists $t_0 \geq 0$ such that for any $t \geq t_0$, $\gamma(t) \cap \bar U =\emptyset$.
Thus for any $t \geq t_0$, $d(x, \gamma(t))-t$ is a viscosity solution of Hamilton-Jacobi equation
$$|\nabla u|_g=1$$
on $U$ [\cite{MM}, Theorem 3.1]. Clearly, two Hamilton-Jacobi equations $|\nabla u|_g=1$ and $|du|^2_g=1$ admit the same set of viscosity solutions. Since $d(x, \gamma(t))-t$ is Lipschitz with Lipschitz constant 1 for ant $t$, $d(x,\gamma(t))-t \rightarrow b_{\gamma}$ locally uniformly. By the stability of viscosity solution [\cite{CS}, Theorem 5.2.5], it means that $b_{\gamma}(x)$ is a viscosity solution of
$$|\nabla u|_g=1 (\text{ or equivalently } |du|^2_g=1)$$
on $U$. By the arbitrariness of $U$, $b_{\gamma}$ is a global viscosity solution of
$$|\nabla u|_g=1(\text{ or equivalently } |du|^2_g=1)$$
on $M$. Since the equation $|du|^2_g=1$ is determined by the locally uniformly convex Hamiltonian $H(x,p)=|p|^2_g$, $b_{\gamma}$, as a viscosity solution of such a  Hamilton-Jacobi equation, must be locally semi-concave with linear modulus [\cite{CS}, Theorem 5.3.6].
\end{proof}

\begin{rem}
By Theorem 1, we know that on any noncompact complete Riemannian manifold $(M,g)$, there always exists at least one viscosity solution to the Hamilton-Jacobi equation
$$|\nabla u|_g=1.$$
In fact, any Busemann function is a such one. One  may ask whether it is still true for closed Riemannian manifolds. The answer is negative.  Otherwise, there exists a viscosity solution $u$ on $M$. Since $M$ is closed, $u$ attains its minimum at some point, say $x_0$. Then, since $u$ is locally semi-concave, $u$ is differentiable at $x_0$, thus $\nabla u(x_0)=0$. It contradicts  the assumption that $u$ is a viscosity solution (Recall that a viscosity solution should satisfies the Hamilton-Jacobi equation at any differentiable point).
\end{rem}

\begin{rem}
By weak KAM theory \cite{Fa}, on a closed  Riemannian manifold  there exists a unique constant $c$ such that
$$|du|^2_g=c$$
admits a global viscosity solution. This  unique constant is characterized  to be $\alpha(0)$, here $\alpha$ is Mather's $\alpha$-function (for the definition, see \cite{Mat1}). In fact, for our case of geodesic flow, $c=0$ and all viscosity solutions of $|du|^2_g$ are constants. For non-compact cases, things are quite different. Although in \cite{FM}, it is proved that there exists a unique $c$ (in our case $c=0$) such that $$|du|^2_g=c(=0)$$ admits a global viscosity solution and for any $c^{\prime} < 0$, $$|du|^2_g=c^{\prime}$$  admits no global viscosity solution, it may happens that $$|du|^2_g=c^{\prime}$$ may admit  global viscosity solution for any $c^{\prime}>0$. Note that Theorem 1 implies that $c^{\prime}$ could be taken to be 1. In fact, for any $c^{\prime}\geq 0$, the Hamilton-Jacobi equation
$$|du|^2_g=c^{\prime}$$
always admits global viscosity solutions. To see this point, choosing any ray $\gamma$, then $\sqrt{c^{\prime}} b_{\gamma}$ would be a such solution. In a word, there may exist more than one constants $c^{\prime}$ such that $$|du|^2_g=c^{\prime}$$
admits global viscosity solutions in the non-compact case, but the constant is unique in the compact case.
\end{rem}
\begin{rem}
By weak KAM theroy, we know that global viscosity solutions are determined essentially by Aubry sets. However, in our case, the elements in $M(\infty)$, which are represented by rays, should be regarded as  analogues of the static classes of Aubry sets. This is a main motivation of this article.
\end{rem}

\section{Singularity sets of Busemann functions}
 In [\cite{In}, Theorem 2], among other results, Innami proved $C(\gamma) \subseteq sing(b_{\gamma})$. By our result on the locally semi-concavity of Busemann functions, we can improve it slightly to Theorem 2.

 Before we going into the details of the proof, we recall some definitions.  For a function $f$ of local semi-concavity (not necessary with  linear modules) and for any $x \in M$, we denote
 $$D^+_xf:=\{\nabla_x \phi : \phi \text{ is  $C^1$  and touch } f \text{ at } x \text{ from above}\}.$$
 If moreover $f$ is assumed to be local semi-concave with linear modulus, here $\phi$ could be taken to be $C^2$ functions.  By the definition of $D^+$, it is easy to see that for any two functions $f_1$ and $f_2$ of local semi-concavity with linear modulus, we have:
 $$D^+(f_1+f_2) \supseteq D^+f_1+D^+f_2.$$
\begin{proof}[Proof of Theorem 2]
Since $b_{\gamma}$ is locally semi-concave,
$$sing(b_{\gamma})=\{x:D^+_xb_{\gamma} \text{ is not a singleton}\},$$
(see for example [\cite{CS}, Proposition 3.3.4]).

Let $D^*b_{\gamma}(x)$ be the reachable gradients of $b_{\gamma}$ at $x$, i.e.
$$D^*b_{\gamma}(x)=\{p:\exists x_k \rightarrow x, b_{\gamma} \text{ is differentiable at $x_k$ and} \lim_{k \rightarrow \infty}\nabla b_{\gamma}(x_k)=p\}.$$
By  [\cite{CS}, Theorem 3.3.6], we have
$$D^+b_{\gamma}(x)=co D^{*}b_{\gamma}(x),$$
for any $x \in M$, here $co$ denotes the convex hull.

Before the  proof of Theorem 2 going on, we need a lemma.

\begin{lem}
For any $p \in D^{*}b_{\gamma}(x)$, there exists at least one coray $\gamma^{\prime}:[0, \infty) \rightarrow M$ with $\gamma^{\prime}(0)=x, \nabla
\gamma^{\prime}(0)=-p.$

\end{lem}
\begin{proof}[Proof of Lemma 2.1]
By definition, there exists a sequence of $x_k$ such that $b_{\gamma}$ is differentiable at $x_k$ and $x_k \rightarrow x,   \nabla b_{\gamma}(x_k) \rightarrow -p$. Since $b_{\gamma}$ is differentiable at $x_k$, there exists a unique coray
$$\gamma_k:[0, +\infty) \rightarrow M $$
to $\gamma$ with $\gamma_k(0)=x_k, \dot{\gamma_k}(0)=-\nabla b_{\gamma}(x_k)$ (by [\cite{In}, Theorem 2], since $C(\gamma) \subseteq sing(b_{\gamma})$ holds or by [\cite{CS}, Theorem 3.3.6] from the viewpoint of viscosity solutions). Then  $\gamma_k$ will convergent uniformly on any compact time interval to a ray $\gamma^{\prime}:[0,+\infty) \rightarrow M$ with $\gamma^{\prime}(0)=x$ and $\dot{\gamma^{\prime}}(0)=-p$. Clearly, $\gamma^{\prime}$ is a coray to $\gamma$.
\end{proof}

Now we continue the proof of Theorem 2. If $x \in sing(b_{\gamma}(x))$, then $D^{+}b_{\gamma}(x)$ is not a singleton, thus $D^{*}b_{\gamma}(x)$ is not a singleton as well. By Lemma 2.1, there exist at least two corays from $x$. So, $x \in C(\gamma)$.

\end{proof}

\section{Proof of Theorem 3}

 Since both $b_{\gamma^{+}}$ and $b_{\gamma^{-}}$ are locally semi-concave with linear modulus, the first assertion of Theorem 3 follows from the obvious fact that the sum of two locally semi-concave functions with linear modulus  are locally semi-concave with linear modulus as well.

Since $B_{\gamma}$ is nonnegative and  locally semi-concave,  $B_{\gamma}$ is differentiable at the points in $G_{\gamma}$ and the differential is zero. Moreover, since $b_{\gamma_+}$ and $b_{\gamma_{-}}$ are locally semi-concave with linear modulus, together with the fact $D_x^+B_{\gamma}\supseteq D_x^+b_{\gamma_+}+D_x^+b_{\gamma_-}$, we know both $b_{{\gamma}_{+}}$ and $b_{\gamma_{-}}$ should be  differentiable simultaneously at the points in $G_{\gamma}$. So there  exist exactly two corays $\gamma^{\prime}_{+}$ and $\gamma^{\prime}_{-}$ to $\gamma_{+}$ and $\gamma_{-}$ respectively initiated from $x \in G_{\gamma}$, with $\dot{{\gamma}^{\prime}_{+}}(0)=-\nabla b_{{\gamma}_{+}}(x)$ and $\dot{{\gamma}^{\prime}_{-}}(0)=-\nabla b_{{\gamma}_{-}}(x)$.

Since
\begin{eqnarray*}
&&0 = \nabla b_{\gamma^+}(x)+\nabla b_{\gamma^-}(x)\\
&\Longleftrightarrow&\dot{{\gamma}^{\prime}_{+}}(0)+\dot{{\gamma}^{\prime}_{-}}(0))=0\\
\end{eqnarray*}
we know that  $\gamma^{\prime}$ defined by
\begin{equation*}\tag{*}
\gamma^{\prime}(t)=
\begin{cases} \gamma^{\prime}_{-}(-t),  \text{ when }  t \leq 0  \\
  \gamma^{\prime}_{+}(t),   \text{ when } t \geq 0
  \end{cases}
\end{equation*}
is really a geodesic.

To show that $\gamma^{\prime}$ is a line, we need to show that
\begin{equation*}\tag{**}
d(\gamma^{\prime}(s), \gamma^{\prime}(s^{\prime}))=|s^{\prime}-s|
\end{equation*}
for any $s^{\prime} \geq s$.
If $s^{\prime} \geq s \geq 0$ or $s \leq s^{\prime} \leq 0$, then clearly $(**)$ holds, since $\gamma^{\prime}_+, \gamma^{\prime}_-$ are rays. Then the only case remained is $s^{\prime} >0 > s$.  Otherwise there exists another geodesic segment $\gamma^{*}:[s,s^{*}] \rightarrow M$ with $\gamma^*(s)=\gamma^{\prime}(s), \gamma^*(s^{*})=\gamma^{\prime}(s^{\prime})$ and $l(\gamma^*|_{[s,s^{*}]})<l(\gamma^{\prime}|_{[s,s^{\prime}]})$. Here, and in the following, $l$ means the length of a curve  induced by the Riemannian metric we considered.

Choose any $t \in (s,s^{*})$, and denote $\gamma^*(t)$ by $y$, we will show that $B_{\gamma}(y) <0$ and thus get a contradiction. In fact,
\begin{eqnarray*}
&&B_{\gamma}(y)\\
&=&b_{\gamma^+}(y)+b_{\gamma^-}(y)\\
&\leq&d(y,\gamma^*(s^{*}))+b_{\gamma^+}(\gamma^*(s^{*}))+d(y,\gamma^*(s))+b_{\gamma^-}(\gamma^*(s))\\
&<&l(\gamma^{\prime}|_{[s,s^{\prime}]})+b_{\gamma^+}(\gamma^*(s^{*}))+b_{\gamma^-}(\gamma^*(s))\\
&=&l(\gamma^{\prime}|_{[s,0]})+l(\gamma^{\prime}|_{[0,s^{\prime}]})+b_{\gamma^+}(\gamma^{\prime}(s^{\prime}))
+b_{\gamma^-}(\gamma^{\prime}(s))\\
&=& s+s^{\prime}+b_{\gamma^+}(\gamma^{\prime}(s^{\prime}))+b_{\gamma^-}(\gamma^{\prime}(s))\\
&=& b_{\gamma^+}(x)+b_{\gamma^-}(x)\\
&=&0.
\end{eqnarray*}

Now we begin to show $B_{\gamma}\gamma^{\prime} \equiv 0$.  We only prove $B_{\gamma}(\gamma^{\prime}(t)) = 0$ for $t \geq 0$ and the case $t <0$ is similar.  In fact,
\begin{eqnarray*}
&&0 \\
&\leq& B_{\gamma}(\gamma^{\prime}(t))\\
&=&b_{\gamma^{+}}(\gamma^{\prime}(t))+b_{\gamma^{-}}(\gamma^{\prime}(t))\\
&\leq& b_{\gamma^+}(\gamma^{\prime}(0))-t+ b_{\gamma^{-}}(\gamma^{\prime}(0))+t\\
&=&b_{\gamma^+}(\gamma^{\prime}(0))+b_{\gamma^-}(\gamma^{\prime}(0))\\
&=&0,
\end{eqnarray*}
where the second inequality holds because of two facts: (1). $\gamma^{\prime}_+$ is a coray to $\gamma^+$, thus by Lemma 0.7, $b_{\gamma^{+}}(\gamma^{\prime}(t))
= b_{\gamma^+}(\gamma^{\prime}(0))-t$ (recall we assume $t\geq 0$); (2) $b_{\gamma^{-}}(\gamma^{\prime}(t))
\leq b_{\gamma^{-}}(\gamma^{\prime}(0))+t$ by the Lipschitz property of $b_{\gamma^-}$. Moreover, it is also easy to see that for $\tau \in \mathbb{R}$, ${\gamma^{\prime}_{\tau}}^+$ is the unique coray to $\gamma^+$ initiated from $\gamma^{\prime}(\tau)$, and ${\gamma^{\prime}_\tau}^-$ is the unique coray to $\gamma^-$ initiated from $\gamma^{\prime}(\tau)$.

Now we see that $G_{\gamma}$ is foliated by lines like $\gamma^{\prime}$, and thus $G_{\gamma}$ forms a lamination.  So far, the second assertion follows.

The third assertion also follows from the local semi-concavity with linear modulus of $B_{\gamma}(x)$. As the first step of the proof, we provide a lemma as follows.

\begin{lem}For any $t_0 \in \mathbb{R}$, there exists a neighborhood $U$ of $\gamma(t_0)$ such that
$$B_{\gamma}(x) \leq C^{\prime}d^2(x, \gamma)$$
for $x \in U$ some constant $C^{\prime}$.
\end{lem}
\begin{proof}[Proof of Lemma 3.1]In fact, since $B_{\gamma}$ is locally semi-concave with linear modulus and obtain minimum at $\gamma(t_0)$, we get $B_{\gamma}$ is differentiable at $\gamma(t_0)$ with differential zero. Choose $U$ to be a sufficiently small neighborhood of $\gamma(t_0)$, such that for any $x \in U$, the minimal geodesic segments connecting $x$ and its foot points on $\gamma$  is contained in a coordinate domain $U^{\prime}$.  Assume the associated coordinate chart is $(U^{\prime}, \phi)$. Shrinking $U$ if necessary (thus $U^{\prime}$ can be chosen  smaller accordingly), we may assume that $B_{\gamma} \circ \phi$ is semi-concave with linear modulus (not a locally one any more) on $\phi(U^{\prime})$.  Assume that the associated semi-concave constant of $B_{\gamma} \circ \phi^{-1}$ is $C_1$. Then we would get: $B_{\gamma} \circ \phi^{-1} \geq 0$ and $B_{\gamma} \circ \phi^{-1}=0$ on $\phi(\gamma(t))$.  Thus, by [\cite{CS}, Proposition 3.3.1], we get for any $x \in U$ and any $\gamma(t) \in U^{\prime}$,
\begin{eqnarray*}
&&B_{\gamma}(x)=B_{\gamma} \circ \phi^{-1}( \phi(x))\\
&=&B_{\gamma} \circ \phi^{-1}( \phi(x))-B_{\gamma} \circ \phi^{-1}( \phi(\gamma(t))) \\
&\leq&  C_1 |\phi(x)-\phi(\gamma(t))|^2\\
&\leq& C_1 |d\phi|^2_{\infty} d^2(x, \gamma(t)).
\end{eqnarray*}

Together with the choice of neighborhood $U$, the above inequality implies $$B_{\gamma}(x) \leq C_1 |d\phi|^2_{\infty}d(x, \gamma).$$
 Now let $C^{\prime}=C_1 |d\phi|^2_{\infty}$ and the lemma follows.
\end{proof}
Now we continue the proof of the third assertion of Theorem 3. By the fact that $\gamma$ is a line, the set of foot points of the compact set $K$ to $\gamma$ will be contained in a  compact segment of $\gamma$, say $\gamma|_{[t_0,t_1]}$. By the finite covering technique, together with Lemma 3.1, we get this fact:
There exist an open neighborhood $U$ of $\gamma|_{[t_0,t_1]}$ and a constant $C^{\prime}$ such that $B_{\gamma}(x) \leq C^{\prime}d^2(x, \gamma)$
for $x \in U$. Let $$K_1 (>0):= \begin{cases}   \inf_{x \in K \setminus U} d(x, \gamma)=K_1, \text{ when }K \setminus U \neq \emptyset \\
 1 ,\text{ when } K \setminus U =\emptyset
  \end{cases},$$
  $K_2 (\geq 0):=\sup_{x \in K} B_{\gamma}(x)$ and  $C=\max{\{C^{\prime}, \frac{K_2}{K^2_1}\}}$, then the assertion of Theorem 3  follows.  \qed

The following corollary can also be easily obtained.
\begin{cor}
 If $B_{\gamma}$ is differentiable at $x$, then both $b_{\gamma^+}$ and $b_{\gamma^{-}}$ are differentiable at $x$. Moreover, there exists a unique pair of  corays $\gamma^{\prime}_{+}, \gamma^{\prime}_{-}$ initiated from $x$ to $\gamma^{+}$ and $\gamma^{-}$ respectively. If moreover, $x$ is a locally minimal point of $B_{\gamma}$, then
 $$\dot{\gamma^{\prime}_{-}}(0)=-\dot{\gamma^{\prime}_{+}}(0).$$
\end{cor}

Alexandroff's theorem [\cite{CS}, Theorem 2.3.1 (i)] says that  locally semi-concave functions with linear modulus are twice differentiable almost everywhere with respect to the Lebesgue measure. With this theorem in hand, we get
\begin{cor}
Busemann functions and Barrier functions are twice differentiable almost everywhere.
\end{cor}
\begin{rem}
In the fields of Riemannian geometry, there are some regularity results for Busemann functions, but almost all of the  results  this type need additional geometrical assumptions (e.g.  negative or positive curvature; free of conjugate points or focal points).  For general case, the regularity of local  semi-concavity with linear modulus (and thus, they are Lipschitz and twice differentiable almost everywhere) is expected to be optimal. Since in general, Busemann functions and barrier functions have no higher regularity, tools from nonsmooth analysis should come into this field essentially.
\end{rem}

\section{On the  relations $\prec$  and $\sim$}
To prove that $\prec$ is a transitive relation, we need the following fundamental  lemma (see for example [\cite{Pe}, Page 286, Proposition 41], [\cite{Mar}, Lemma 2.7]).
\begin{lem}
If $\gamma$ is a ray, and $\gamma^{\prime}$ is a coray  to $\gamma$, then $b_{\gamma}(x)-b_{\gamma}(\gamma^{\prime}(0)) \leq b_{\gamma^{\prime}}(x)$.
\end{lem}
This lemma plays a crucial role in our paper. For the completeness, we provide the proof here.
\begin{proof}[Proof of Lemma 4.1]
For any fixed $x \in M$ and any $\epsilon >0$, there exists $T(\epsilon)$ such that
$$d(x, \gamma^{\prime}(T))\leq T +b_{\gamma^{\prime}}(x)+\frac{\epsilon}{3}.$$
Since $\gamma^{\prime}$ is the limit of some sequences of minimal geodesic segments $\gamma_n $ which join
  $a_n \rightarrow \gamma^{\prime}(0)$ and $\gamma(t_n)$ for some sequence  $t_n \rightarrow \infty$,  for given $\epsilon$, there exists $N(\epsilon, T)$ such that for $n \geq N$, we have
$$d(a_n, \gamma^{\prime}(0)) < \frac{\epsilon}{3}$$
and
$$d(\gamma_{n}(T), \gamma^{\prime}(T))< \frac{\epsilon}{3}.$$

Thus, we obtain
\begin{eqnarray*}
&&d(x,\gamma(t_n)) \\
&\leq& d(x, \gamma_n(T))+d(\gamma_n(T), \gamma(t_n))\\
&=& d(x, \gamma_n(T))+d(a_n, \gamma(t_n))-T\\
&<& d(x,\gamma^{\prime}(T))+\frac{\epsilon}{3}+
d(\gamma^{\prime}(0),\gamma(t_n))+\frac{\epsilon}{3}-T\\
&\leq& T+b_{\gamma^{\prime}}(x)+\frac{2\epsilon}{3}+d(\gamma^{\prime}(0),\gamma(t_n))
+\frac{\epsilon}{3}-T\\
&=& \epsilon+b_{\gamma^{\prime}}(x)+d(\gamma^{\prime}(0),\gamma(t_n)),
\end{eqnarray*}
here, the first equality holds because $\gamma_n$ are minimal geodesic segments. Thus,
$$ d(x,\gamma(t_n))-d(\gamma^{\prime}(0),\gamma(t_n)) \leq \epsilon+b_{\gamma^{\prime}}(x)$$
for all $n > N(\epsilon, T)$. Passing to the limit as $\epsilon \rightarrow 0$, we obtain the assertion of the lemma.
\end{proof}

By Lemma 4.1, we can easily get
\begin{prop}
If $\gamma^{\prime} \prec \gamma$, then $B_{\gamma} \leq B_{\gamma^{\prime}}$. Consequently, if $\gamma^{\prime} \prec \gamma$ and $\gamma \prec \gamma^{\prime}$ (i.e. $\gamma \sim \gamma^{\prime}$), then $B_{\gamma} = B_{\gamma^{\prime}}$.
\end{prop}
\begin{proof}[Proof of proposition 4.2]
If $\gamma^{\prime} \prec \gamma$, then
\begin{eqnarray*}
&&B_{\gamma^{\prime}}(x)\\
&=&b_{{\gamma^{\prime}}^+}(x)+b_{{\gamma^{\prime}}^-}(x)\\
&\geq& b_{\gamma^+}(x)-b_{\gamma^+}(\gamma^{\prime}(0))+b_{\gamma^-}(x)-
b_{\gamma^-}(\gamma^{\prime}(0))\\
&=&b_{\gamma^+}(x)+b_{\gamma^-}(x)\\
&=&B_{\gamma}(x),
\end{eqnarray*}
here, the second equality holds because $0=B_{\gamma}(\gamma^{\prime}(0))=b_{\gamma^+}(\gamma^{\prime}(0))
+b_{\gamma^-}(\gamma^{\prime}(0)) $
\end{proof}

\begin{rem}
Note that for any two lines $\gamma$ and $ \gamma^{\prime}$, $B_{\gamma}=B_{\gamma^{\prime}}$ dose not imply that $\gamma^{\prime} \prec \gamma$ and/or $\gamma \prec \gamma^{\prime}$, just by recalling that on Euclidean space $\mathbb{R}^n$ for any line $\gamma, B_{\gamma} \equiv 0$.
\end{rem}

Now we show that  $\prec$  is indeed a transitive relation.
\begin{prop}
If $\gamma_2\prec \gamma_1, \gamma_3 \prec \gamma_2$, then $\gamma_3 \prec \gamma_1$.
\end{prop}
\begin{proof}
First we will prove that $B_{\gamma_1}(\gamma_3)\equiv 0$. For any $\tau \in \mathbb{R}$, denote $\gamma_3(\tau)$ by $x$. Since $B_{\gamma_2}(x)=0$, for any $\epsilon >0$, there exists $T(\epsilon)>0$ such that for $t \geq T$,
$$0 \leq d(x, \gamma_2(t))+d(x, \gamma_2(-t))-2t \leq \frac{\epsilon}{4}. $$
Denote $\gamma_2(0)$ by $y$.  Since $B_{\gamma_1}(y)=0$, there exists $S(\epsilon, T)>0$ such that for $s \geq S$,
$$0 \leq d(y, \gamma_1(s))+d(y,\gamma_1(-s))-2s \leq \frac{\epsilon}{4},$$
and minimal geodesic segments $\gamma_{\pm s}$  with
$\gamma_{\pm s}(0)=y, \gamma_{\pm s}(d(y,\gamma_1({\pm s})))=\gamma_1({\pm s})$,
$$d(\gamma_{s}(T), \gamma_2(T))< \frac{\epsilon}{4} \text{ for } s \geq S,$$
$$d(\gamma_{- s}(T), \gamma_2(-T))< \frac{\epsilon}{4} \text{ for } s \geq S$$
hold. Here, since $B_{\gamma}(y)=0$, $\gamma^+_2$ and $\gamma^-_2$ are the only two corays issued from $y$ to $\gamma^+_1$ and $\gamma^+_2$ respectively, $\gamma_{\pm s}(0)$ could be fixed at $y$.
Thus, for $ s \geq S$,
\begin{eqnarray*}
&&0\\
&\leq& d(x, \gamma_1(s))+d(x,\gamma_1(-s))-2s\\
&\leq& d(x, \gamma_2(T))+d(\gamma_s(T), \gamma_1(s))+\frac{\epsilon}{4}+ d(x, \gamma_2(-T))+d(\gamma_{-s}(T), \gamma_1(-s))+\frac{\epsilon}{4}-2s\\
&\leq& 2T+\frac{\epsilon}{4}+d(y, \gamma_1(s))-d(y, \gamma_s(T))+d(y, \gamma_1(-s))-d(y, \gamma_{-s}(T))-2s+\frac{\epsilon}{2}\\
&\leq& 2T+\frac{\epsilon}{4}+2s+\frac{\epsilon}{4}-2T
-2s+\frac{\epsilon}{2}\\
&=& \epsilon.
\end{eqnarray*}
Since $\epsilon$ can be arbitrarily chosen, we get $B_{\gamma_1}(x)=0$. Since $x=\gamma_3(\tau)$ and $\tau$ can be arbitrarily chosen, we get $B_{\gamma_1}(\gamma_3)\equiv 0$.

Now will prove that that for any $\tau$, ${{\gamma_3}_{\tau}}^+$ is a coray to $\gamma^+_1$. In fact, by Lemma 4.1, for ant $t\in \mathbb{R}$,
$$b_{\gamma^+_1}({\gamma_3}_{\tau}(t)) \leq b_{\gamma^+_1}({\gamma_2}(0))+b_{\gamma^+_2}({\gamma_3}_{\tau}(t)).$$
Similarly,
$$b_{\gamma^-_1}({\gamma_3}_{\tau}(t)) \leq b_{\gamma^-_1}({\gamma_2}(0))+b_{\gamma^-_2}({\gamma_3}_{\tau}(t)).$$
By discussions above, together with facts $\gamma_2\prec \gamma_1, \gamma_3 \prec \gamma_2,$ we get
\begin{eqnarray*}
&&0\\
&=& b_{\gamma^+_1}({\gamma_3}_{\tau}(t))+b_{\gamma^-_1}({\gamma_3}_{\tau}(t))\\
&\leq& b_{\gamma^+_1}({\gamma_2}(0))+b_{\gamma^-_1}({\gamma_2}(0))\\
&&+ b_{\gamma^+_2}({\gamma_3}_{\tau}(t))+b_{\gamma^-_2}({\gamma_3}_{\tau}(t))\\
&=&0.
\end{eqnarray*}

So, we in fact get
$$b_{\gamma^+_1}({\gamma_3}_{\tau}(t)) = b_{\gamma^+_1}({\gamma_2}(0))+b_{\gamma^+_2}({\gamma_3}_{\tau}(t))$$
and
$$b_{\gamma^-_1}({\gamma_3}_{\tau}(t)) = b_{\gamma^-_1}({\gamma_2}(0))+b_{\gamma^-_2}({\gamma_3}_{\tau}(t)).$$

Thus for any $t_1, t_2 \geq 0$,
$$b_{\gamma^+_1}({\gamma_3}_{\tau}(t_2)) = b_{\gamma^+_1}({\gamma_2}(0))+b_{\gamma^+_2}({\gamma_3}_{\tau}(t_2)),$$
$$b_{\gamma^+_1}({\gamma_3}_{\tau}(t_1)) = b_{\gamma^+_1}({\gamma_2}(0))+b_{\gamma^+_2}({\gamma_3}_{\tau}(t_1)).$$

Now we have
\begin{eqnarray*}
&&b_{\gamma^+_1}({\gamma_3}_{\tau}(t_2))-b_{\gamma^+_1}({\gamma_3}_{\tau}(t_1))\\
&=&b_{\gamma^+_2}({\gamma_3}_{\tau}(t_2))-b_{\gamma^+_2}({\gamma_3}_{\tau}(t_1))\\
&=&t_1-t_2.
\end{eqnarray*}
By Lemma 0.7, it implies that ${{\gamma_3}_{\tau}}^+$ is a coray to $\gamma^+_1$ for any $\tau$. Similarly, ${{\gamma_3}_{\tau}}^-$ is a coray to $\gamma^-_1$ for any $\tau$.

So far, we have proved that $\gamma_2\prec \gamma_1, \gamma_3 \prec \gamma_2$ implies $\gamma_3 \prec \gamma_1$.

\end{proof}
\begin{rem}
Proposition 4.4 says that $\prec$ is a translative relation, and thus $\sim$ is an equivalence relation.
\end{rem}

 Barrier function in essence is the sum of two Busemann functions associated two rays which determined by a line. One may ask how about the sum of two Busemann functions associated two rays initiated from the same point in general. In fact, we have

 \begin{prop}
 Assume that $\gamma_1, \gamma_2$ are two rays with $\gamma_1(0)=\gamma_2(0)$, then  $b_{\gamma_1}+b_{\gamma_2} \geq 0$ if and only if the curve $\gamma^{\prime}$ defined by \begin{equation*}
\gamma^{\prime}(t)=
\begin{cases} \gamma_2 (-t),  \text{ when }  t \leq 0  \\
  \gamma_1(t),   \text{ when } t \geq 0
  \end{cases}
\end{equation*} is a line.
 \end{prop}
 \begin{proof}[Proof of Proposition 4.6]
The direction ``$\Leftarrow$'' is obvious.

Now we prove the other direction. Otherwise,  $\gamma^{\prime}$ is not a line, i.e. there exist $t_1,t_2>0$ such that
$$d(\gamma_1(t_1), \gamma_2(t_2))< t_1+t_2.$$
Then $$t_1+t_2>d(\gamma_1(t_1), \gamma_2(t_2)) \geq b_{\gamma_2}(\gamma_1(t_1))-b_{\gamma_2}(\gamma_2(t_2))=b_{\gamma_2}(\gamma_1(t_1))+t_2.$$
Namely, $$t_1> b_{\gamma_2}(\gamma_1(t_1)).$$
Since $b_{\gamma_1}(\gamma_1(t_1))=-t_1$, we obtain
$$b_{\gamma_1}(\gamma_1(t_1))+b_{\gamma_2}(\gamma_1(t_1)) < t_1+(-t_1)=0,$$
this contradicts the assumption.
 \end{proof}
 For two lines $\gamma_1$ and $\gamma$ with $\gamma_1 \prec \gamma$, we denote

$$S_{\gamma_1, \gamma}=\left\{ \gamma^{\prime} \text{ are lines} : \gamma^{\prime} \prec \gamma \text{ and } B_{\gamma^{\prime}}=B_{\gamma_1} \right\}.$$

Then we have the following proposition (comparing with Remark 4.3).
\begin{prop}
We fix a line $\gamma$. For other two lines $\gamma_1$ and $\gamma_2$ with $\gamma_1 \prec \gamma, \gamma_2 \prec \gamma$,
 $B_{\gamma_1}\geq B_{\gamma_2}$ implies $\gamma_1 \prec \gamma_2$. In this case (i.e.  $\gamma_1 \prec \gamma, \gamma_2 \prec \gamma$), we have $\gamma_1\sim \gamma_2$ if and only if  $B_{\gamma_1}=B_{\gamma_2}$.
\end{prop}

\begin{proof}[Proof of Proposition 4.7]
 If $B_{\gamma_1}\geq B_{\gamma_2}$, then we get $B_{\gamma_2}(\gamma_1) \equiv 0$. If $\gamma_1\nprec \gamma_2$, then by the second assertion of Theorem 3, there exists another line $\xi$, such that $\xi(0)=\gamma_1(0)$ and $\xi \prec \gamma_2$. By the transitivity of $\prec$, $\xi \prec \gamma$. Thus, there are two lines $\xi$ and $\gamma_1$ such that $\xi \prec \gamma$ and $\gamma_1 \prec \gamma$, it will contradict the uniqueness property in the second assertion of Theorem 3.
\end{proof}

Combing Proposition 4.4 and Proposition 4.7, we complete the proof of Theorem 4.

Proposition 4.7 could be strengthened to
\begin{prop}
For two lines $\gamma_1$ and  $\gamma_2$ with $\gamma_1 \prec \gamma, \gamma_2 \prec \gamma$, if there exists a $t_0 \in \mathbb{R}$ such that
 $B_{\gamma_1}(\gamma_2(t_0))=0$, then $\gamma_2 \prec \gamma_1$.
\end{prop}
\begin{proof}[Proof of Proposition 4.8]
Otherwise, there exists another line $\xi$ with $\xi(0)=\gamma_2(t_0)$ such that
$\xi \prec \gamma_1$. Also by the transitivity of $\prec$,  there exists two lines $\xi$ and $\gamma_{2_{t_0}} $ through $\gamma_2(t_0)$ with  both  $\xi \prec \gamma$ and $\gamma_{2_{t_0}} \prec \gamma$. It will also contradict the uniqueness property in the second assertion of Theorem 3.
\end{proof}

\section{Geometric meaning of the relation $\sim$}

  For a line $\gamma_1$, we may think it as a geodesic connecting two elements (i.e. $\gamma^-, \gamma^+$ ) in $M(\infty)$. For any other line $\gamma_2$,  $\gamma_2$ connects the same two boundary elements as $\gamma_1$ if and only if  $$b_{\gamma^{\pm}_1}=b_{\gamma^{\pm}_2}+const..$$ Given two lines $\gamma_1$ and $\gamma_2$, the following proposition shows that  relation $ \gamma_1 \sim \gamma_2$ implies nothing but that $\gamma_1$ and $\gamma_2$ connect the same pair of elements in $M(\infty)$.

 \begin{prop}
 For two lines $\gamma_1$ and $\gamma_2$, $b_{\gamma^{\pm}_1}=b_{\gamma^{\pm}_2}+const. $ hold if and only $\gamma_1 \sim \gamma_2$.
 \end{prop}
 \begin{proof}[Proof of Proposition 5.1]
$\Longrightarrow).$ Assume that $b_{\gamma^{+}_1}=b_{\gamma^{+}_2}+c_1$ and $b_{\gamma^{-}_1}=b_{\gamma^{-}_2}+c_2$ for two constants $c_1,c_2$, then we get $B_{\gamma_1}=B_{\gamma_2}+c_1+c_2.$ By $B_{\gamma_1}(\gamma_1(0)) =0$ and $B_{\gamma_2}(\gamma_1(0)) \geq 0$, we get $c_1+c_2 \leq 0$.  Analogously, by $B_{\gamma_2}(\gamma_2(0)) =0$ and $B_{\gamma_1}(\gamma_2(0)) \geq 0$, we get $c_1+c_2 \geq 0$.  Thus, we get $c_1+c_2=0$ and consequently  $B_{\gamma_1}=B_{\gamma_2}$. So, we get $B_{\gamma_2}(\gamma_1(0))=0$.  As a consequence, $B_{\gamma_2}, b_{\gamma^+_2},b_{\gamma^-_2}$ are differentiable at $\gamma_1(0)$.  Also by the assumption, together with Lemma 0.7, $\gamma^+_1$ is a coray to $\gamma^+_2$ and $\gamma^-_1$ is a coray to $\gamma^-_2$.  If $\gamma_1 \nprec \gamma_2$, then there exists another line $\xi$ with $\xi(0)=\gamma_1(0)$ such that $\xi \prec \gamma_2$, (by Theorem 3. 2)). Thus we get both $\xi^+ $ and $\gamma_1^+$ are corays to $\gamma_2^+$, and both $\xi^-$ and $\gamma_1^-$ are corays to $\gamma_2^-$.  It  impossible since it contradicts the differentiability of $b_{\gamma^+_2}$ and $b_{\gamma^-_2}$  at $\gamma_1(0)$. So we have $\gamma_1 \prec \gamma_2$. The proof of  $\gamma_2 \prec \gamma_1$ is similar. Thus, we obtain $\gamma_1 \sim \gamma_2$.

$\Longleftarrow).$  By Lemma 4.1, for any $x \in M$
$$b_{\gamma^+_1}(x) \geq b_{\gamma^+_2}(x) -b_{\gamma^+_2}(\gamma_1(0))$$
and
$$b_{\gamma^-_1}(x) \geq b_{\gamma^-_2}(x) -b_{\gamma^-_2}(\gamma_1(0)).$$
So, we obtain
\begin{eqnarray*}
&&B_{\gamma_1}(x)\\ &\geq& b_{\gamma^+_2}(x) -b_{\gamma^+_2}(\gamma_1(0)) + b_{\gamma^-_2}(x) -b_{\gamma^-_2}(\gamma_1(0))\\
&=& B_{\gamma_2}(x)-B_{\gamma_2}(\gamma_1(0))\\
&=&B_{\gamma_2}(x)\\
&=&B_{\gamma_1}(x).
\end{eqnarray*}
Thus, we get two equalities
$$b_{\gamma^+_1}(x) = b_{\gamma^+_2}(x) -b_{\gamma^+_2}(\gamma_1(0)), \,\,\,  b_{\gamma^-_1}(x) =b_{\gamma^-_2}(x) -b_{\gamma^-_2}(\gamma_1(0)).$$
 \end{proof}

Theorem 5 is just a restatement of Proposition 5.1.

\begin{rem} As we said in Remark 1.4, in some sense we could regard the elements in $M(\infty)$ as some kind of analogy of static classes of Aubry sets  in  positive definite Lagrangian systems (for Mather theory, we refer to \cite{Mat1}, \cite{Mat2}, \cite{Fa}). Hence, motivated by the study of connecting orbits in positive definite Lagrangian systems \cite{Mat2}, to construct connecting geodesics (i.e. a geodesic $\gamma: \mathbb{R} \rightarrow M$ such that $\gamma|[T, \infty)$ and $-\gamma|[T, \infty)$ are two rays for sufficiently large $T>0$, but $\gamma$ is maybe not  a line) should be very interesting.
\end{rem}

\begin{rem}
For any fixed line $\gamma$, we may reparameterize all lines  $\gamma^{\prime}$ with $\gamma^{\prime} \prec \gamma$ such that $\gamma^{\prime}(0) \in b^{-1}_{\gamma^+}(0)$. Then for any two lines $\gamma_1,\gamma_2$ with $\gamma_1 \prec \gamma, \gamma_2 \prec \gamma$,  we can define $d_{\gamma}(\gamma_1, \gamma_2)=B_{\gamma_1}(\gamma_2(0))+B_{\gamma_2}(\gamma_1(0))$. By Proposition 4.8, we obtain $d_{\gamma}(\gamma_1,\gamma_2)=0$ if and only if $\gamma_1 \sim \gamma_2$. But it is not clear to us  whether $d_{\gamma}$ is a pseudo-metric on the set of lines $\gamma^{\prime}$ with  $\gamma^{\prime} \prec \gamma$.
\end{rem}

\section{Rigidity conjectures}
In the field of differential geometry, there is a well known rigidity conjecture due to G. Knieper, who stated it in a conference at MSRI in 1991. The conjecture also appeared in  literature (e.g. {\cite{BK}, \cite{Cr}, \cite{BE2}}).
\begin{conj}[Knieper]
On a complete Riemannian plane $(\mathbb{R}^2,g)$, if for every geodesic $\gamma$ and any point $x$ outside $\gamma$, there exists a unique geodesic through $x$ does not intersect $\gamma$, then $g$ is flat.
\end{conj}

Clearly, the conjecture cannot generalize to higher dimensions directly, since even in the Euclidean space $\mathbb{R}^n (n \geq 3)$, for every geodesic (i.e. straight line) $\gamma$ and any point $x \notin \gamma$, there exist infinitely many geodesics through $p$ do not intersect $\gamma$.

However, it seems plausible to  generalize this conjecture to higher dimensions as follows:
\begin{conj}
On a complete Riemannian  $n$-plane $(\mathbb{R}^n,g)$, if every geodesic is a line and $S_{\gamma,\gamma} $ (i.e. the set of $\gamma^{\prime}$ with $\gamma^{\prime} \sim \gamma$) foliates $\mathbb{R}^n$ for any line $\gamma$, then $g$ is flat.
\end{conj}

\section{Further discussions}
On a  noncompact Riemannian metric $(M,g)$, there are other two \emph{generalized} Busemann functions, defined as follows.

Let ${x_n}$ be a sequence of points in $M$ such that $d(y,x_n) \rightarrow \infty$ for some fixed point $y$ (hence for any other fixed point in $M$) and  $$h(x):= \lim[d(x,x_n)-d(y,x_n)]$$ exists in the compact-open topology. Such a limit function will be called horofunction. More general, let ${K_n}$ be a sequence of closed subsets in $M$ such that $d(y,K_n) \rightarrow \infty$ for some fixed point $y$ (hence for any other fixed point in $M$) and  $$h(x):= \lim[d(x,K_n)-d(y,K_n)]$$ exists in the compact-open topology. Such a limit function will be called $dl$ (distance like)-function.

\begin{rem}
Clearly, any Busemann function is a horofunction and any horofunction is a $dl$-function, but the converses are not true any more. For example, in Remark 0.6, the function $-|x|$ is a $dl$-function, but not a Busemann function. [\cite{Mar}, Example 1.6] provides an example where a function is $dl$-function but not a horofunction.
\end{rem}

By the same argument  as Busemann functions, we can obtain the following result.

\begin{cor}
All horofunctions and $dl$-functions are viscosity solutions with respect to the Hamilton-Jacobi equation $|\nabla u |_g=1$ and thus locally semi-concave with linear modulus.
\end{cor}

As it is  explained explicitly in \cite{Mar}, we could also use $$\{\text{horofunctions}\}/\{\text{constant functions}\}$$ and $$\{dl\text{-functions}\}/\{\text{constant functions}\}$$ to define other kinds of ideal boundary. From the view point of Mather theory, the one $M(\infty)$, defined by rays (or equivalently by Busemann functions) is more reasonable, although in Gromov's theory (e.g. [\cite{Gr}, 1.2]), the one defined by horofunctions is more suitable for the name ``ideal boundary''.

One could pose such an interesting problem:

\begin{prob}
Whether any viscosity solution to the Hamilton-Jacobi equation
$$|\nabla u|_g=1$$  must be a $dl$-function?
\end{prob}

\begin{ack}
We would like to thank Professor G. Knieper for some remarks on the conjecture 6.1.
\end{ack}


\begin{thebibliography}{99}
\bibitem{Ba}W. Ballmann, \textsl{Lectures on spaces of nonpositive curvature},
   DMV Seminar, 25, Birkh\"{a}user, (1995).

\bibitem{BE1}V. Bangert $\&$ P. Emmerich, \textsl{On the flatness of Riemannian cylinders without conjugate points}, Comm. Anal. Geom., 19 (4), (2011), 773-805.

\bibitem{BE2}V. Bangert $\&$ P. Emmerich, \textsl{Area growth and rigidity of surface without conjugate points}, J. Diff. Geom.,  94(3), (2013), 367-385.

\bibitem{BG} V. Bangert $\&$ E. Gutkin, \textsl{Insecurity for compact surfaces of positive genus},  Geometriae Dedicata,  146, (2010), 165-191.

\bibitem{BK} K Burns $\&$ G Kniper, \textsl{Rigidity of surfaces with no conjugate points},  J. Diff. Geom., 34, (1991), 623-650.

\bibitem{Bu} H. Busemann, \textsl{The geometry of geodesics}, Dover Publications, INC,  (1995).

\bibitem{CS}P. Cannarsa $\&$ C. Sinestrari, \textsl{Semiconcave functions, Hamilton-Jacobi equations, and optimal control},  Birkh\"{a}user,   (2004).

\bibitem{Cr}C. B. Croke, \textsl{ A synthetic characterization of the hemisphere}, Proc. Amer. Math. Soc. 136 (3), (2008), 1083-1086.

\bibitem{Es}J. Eschenburg, \textsl{ Horospheres and the stable part of the geodesic flow}, Math. Z.,  153, (1977), 237-251.

\bibitem{Fa} A. Fathi, \textsl{ Weak KAM theorem in Lagrangian dynamics}, Preprint.

\bibitem{FM} A. Fathi $\&$ E. Maderna, \textsl{ Weak KAM theorem on non-compact manifolds}, Nonlinear Differential Equations Appl., (14), (2007), 1-27.

\bibitem{Gr} M. Gromov, \textsl{ Hyperbolic manifolds, groups and actions}, Ann. of Math. Stud., Princeton, 97, (1981), 183-215.

\bibitem{HiH} E. Heintze $\&$  H-C Im Hof, \textsl{Geometry of horospheres}, J. Diff.  Geom.,  12(4), (1977), 481-491.

\bibitem{In} N. Innami, \textsl{Differentiablity of Busemann functions and total excess}, Math. Z., 180, (1982), 235-247.

\bibitem{Li} P. L. Lions, \textsl{Generalized solutions of Hamilton-Jacobi equations}, Pitman, (1982).

\bibitem{MM} C. Mantegazza $\&$ A. C. Mennucci, \textsl{Hamilton-Jacobi equations and distance functions on Riemannian manifolds}, Appl. Math. Optim,  47, (2003), 1-25.

\bibitem{Mar} V. B. Marenich, \textsl{Horofunctions, Busemann functions, and ideal boundaries of open manifolds of nonnegative curvature},  Siberian Mathematical Journal, 34(5), (1993), 883-897.

\bibitem{Mat1} J. Mather, \textsl{Action minimizing invariant measures for positive definite Lagrangian systems}, Math. Z., 207(2), (1991), 169-207.

\bibitem{Mat2} J. Mather, \textsl{Variational constructions of connecting orbits}, Ann. Inst. Fourier, Grenoble, 43(5), (1993), 1349-1386.

\bibitem{Pe} P. Petersen,  \textsl{Riemannian geometry},  Springer, (2006).

\bibitem{Sa} T.  Sakai,  \textsl{Riemannian geometry}, American Mathematical Society, (1996).
\bibitem{Wu} H. Wu, \textsl{An enelentary method in the study of nonnegative curvature}, 142(1), (1979), 57-78.

\end{thebibliography}
\end{document}